\documentclass{amsart}
\usepackage{latexsym, bbm, html, enumerate, amssymb, amsmath}
\usepackage[abbr,dcucite]{harvard}
\usepackage[dvips]{graphicx}
\usepackage{setspace}

\usepackage[show]{ed}

\usepackage{tikz}

\theoremstyle{plain}
\newtheorem{theorem}{Theorem}

\newtheorem{lemma}[theorem]{Lemma}
\newtheorem{proposition}[theorem]{Proposition}
\theoremstyle{definition}

\allowdisplaybreaks[1]

\renewcommand{\P}{\mathbb{P}}

\newcommand{\E}{\mathbb{E}}
\newcommand{\FF}{\mathcal{F}}

\newcommand{\x}{\boldsymbol{x}}

\newcommand{\1}{\mathbbm{1}}

\DeclareMathOperator{\Var}{Var}

\let\oldmarginpar\marginpar
\renewcommand\marginpar[1]{\-\oldmarginpar[\raggedleft\footnotesize #1]%
{\raggedright\footnotesize #1}}

\begin{document}

\title[On-Line Alternating Subsequences]%
{On-Line Selection of Alternating Subsequences from a Random Sample}
\author[Arlotto, A., Chen, R.W., Shepp, L.A., Steele, J. M.]
{Alessandro Arlotto, Robert W. Chen,\\ Lawrence A. Shepp and J. Michael Steele}

\thanks{A. Arlotto:  Wharton School, Department of Operations and Information Management, Huntsman Hall
527.2, University of Pennsylvania, Philadelphia, PA 19104}

\thanks{R.W. Chen: Department of Mathematics, University of Miami, Coral Gables, FL 33124}

\thanks{L.A. Shepp:  Wharton School, Department of Statistics, Huntsman Hall
462, University of Pennsylvania, Philadelphia, PA 19104}

\thanks{J.M.
Steele:  Wharton School, Department of Statistics, Huntsman Hall
447, University of Pennsylvania, Philadelphia, PA 19104}

\begin{abstract}
We consider sequential selection of an alternating subsequence from a sequence of independent, identically distributed, continuous random variables,
and we determine the exact asymptotic behavior of
an optimal sequentially selected subsequence. Moreover, we find (in a sense we make precise)
that a person who is constrained to make sequential selections does only about 12\% worse than a
person who can make selections with full knowledge of the random sequence.



\end{abstract}


\maketitle


\section{Introduction}

Given a finite (or infinite)
sequence  $\x = \{x_1,x_2,...,x_n, \ldots\}$ of real numbers, we say that a subsequence
$x_{i_1}, x_{i_2}, \ldots, x_{i_k}, \ldots$ with
$1\leq i_1 < i_2 < \ldots < i_k < \cdots $
is \emph{alternating} if we have
$
x_{i_1} < x_{i_2}> x_{i_3} <  x_{i_4 }\cdots.
$
When $\x$ is an element of the set of permutations $\mathcal{S}_n$ of the  integers $\{1,2,\ldots, n\}$, the study of the set of
alternating permutations goes back to Euler \citeaffixed{Sta:CM2010}{c.f.}.

Here we are mainly concerned with the length $a(\x)$ of the longest alternating subsequence of $\x$. This function has
been more recently studied by \citeasnoun{Wid:EJC2006}, Pemantle \citeaffixed[p.
568]{Sta:PROC2007}{c.f.} and \citeasnoun{Sta:MMJ2008}.
In particular, they consider the
situation in which $\x$ is chosen at random from $\mathcal{S}_n$. By exploiting explicit
formulas for generating functions and delicate applications of the saddle point method, they were able to obtain exact
formulas for the first two moments and to prove a central limit theorem. Specifically, if
$\x$ is chosen according to the uniform distribution on the set of permutations  $\mathcal{S}_n$ and if
$A_n:= a(\x)$ denotes the length of the longest alternating subsequence of $\x$, then for $n\geq 4$ one has
\begin{eqnarray*}
\E[A_n]  = \frac{2n}{3} + \frac{1}{6} \quad \text{ and }
\Var[A_n]  = \frac{8n}{45} - \frac{13}{180}.
\end{eqnarray*}

More recently, \citeasnoun{HouRes:EJC2010} used purely probabilistic means to obtain a simpler proof of this result and the corresponding
central limit theorem. Moreover,
the methods of \citename{HouRes:EJC2010} also apply to  models of random words that are more refined than simple random selection from set of permutations.

Here, we study the problem of making \emph{on-line selection} of an alternating subsequence. That is, we now regard the
sequence $x_1, x_2, ...$ as being presented to us sequentially, and, at the time $i$ when  $x_i$ is presented, we
must choose to include $x_i$ as a term of our subsequence ---
or we must reject $x_i$ as a member of the subsequence.

We will consider the sequence to be given
by independent random variables $X_1,X_2,\ldots $ that have a common continuous distribution $F$, and, since we are only concerned with order
properties, one can without loss of generality take the distribution to be uniform on $[0,1]$.
We now need to be more explicit about the set $\Pi$ of feasible strategies  for on-line selection.
At time $i$, when presented with $X_i$ we must decide to select $X_i$ based on its value,
the value of earlier members of the sequence, and the actions we have taken in the past. All of this information
can be captured by saying that
$\tau_k$, the index of the $k$'th selection, must be a stopping time with respect to the increasing
sequence of $\sigma$-fields,
$
\FF_i = \sigma\{X_1, X_2, \ldots, X_i\}, \, i=1,2, \ldots.
$
Given any feasible policy $\pi \in \Pi$ the random variable of most interest here is
$A^o_n(\pi)$, the number of selections made by the policy $\pi$ up to and including time $n$.
In other words, $A^o_n(\pi)$ is equal to the largest $k$ for which there are stopping times
$1\leq \tau_1 < \tau_2 < \cdots < \tau_k \leq n$ such that
$\{X_{\tau_1}, X_{\tau_2}, \ldots, X_{\tau_k} \}$ is an alternating sequence.

\begin{theorem}[Asymptotic Selection Rate for Large Samples]\label{th1:fixed-mean}
For each $n=1,2,...$, there is a policy $\pi^*_n\in \Pi$ such that
\begin{equation*}
\E[A^o_n(\pi^*_n)] = \sup_{\pi\in \Pi}\E[A^o_n(\pi)],
\end{equation*}
and for such an optimal policy one has for all $n\geq 1$ that
\begin{equation*}
 (2 - \sqrt{2}) n \leq   \E[A^o_n(\pi^*_n)]  \leq (2 - \sqrt{2}) n + C,
\end{equation*}
where $C$ is a constant with $C < 11 - 4\sqrt{2} \sim 5.343.$ In particular, one has
\begin{equation*}
   \E[A^o_n(\pi^*_n)]  \sim  (2 - \sqrt{2}) n \quad \hbox{as $n\rightarrow \infty$}.
\end{equation*}
\end{theorem}

The proof of this result exploits the analysis of a closely related
selection problem in which one considers a sample of size $N$ where $N$ is geometrically distributed
with parameter $0 < \rho < 1$ (so one has
$
\P( N = k ) = \rho^{k-1}(1 - \rho), \, k = 1,2,3,\ldots.
$)
Here we also assume that $N$ is independent of the sequence  $X_1,X_2,\ldots $.

Parallel to our first theorem, we consider the number $A^o_N(\pi)$ of selections made by a feasible policy $\pi$ up to and
including the random time $N$.
The geometric smoothing provided by $N$ gives us a useful ``shift symmetry" that is missing in the fixed $n$ problem, and
the analysis of a geometric sample turns out to be far more tractable. In particular,
one can determine the \emph{exact} expected length of the sequence selected by an optimal policy.

\begin{theorem}[Expected Selection Size in Geometric Samples]\label{th2:geometric-mean}
For each $0< \rho < 1$, there is a $\pi^*\in \Pi$, such that
\begin{equation*}
\E[A^o_N(\pi^*)] = \sup_{\pi\in \Pi}\E[A^o_N(\pi)],
\end{equation*}
and for such an optimal policy one has
\begin{equation*}
   \E[A^o_N(\pi^*)]  =\frac{3 - 2\sqrt{2} - \rho + \rho\sqrt{2}}{\rho (1 - \rho)}
   \sim  (2 - \sqrt{2}) (1 - \rho)^{-1} \quad \hbox{as $\rho\rightarrow 1$}.
\end{equation*}
\end{theorem}

These theorems respectively tell us that optimal on-line selection yields subsequences that grow at a linear rate
$(2-\sqrt{2})n \sim 0.585\, n $ or $(2-\sqrt{2}) \E N \sim 0.585 \E N$.
This is about a 12\% discount off the rate $(2/3)n \sim 0.667 \, n$ that one would obtain with \emph{a priori}
knowledge of the full finite sample $\{X_1,X_2, ..., X_n\}$,
and this discount seem quite modest given the great difference in the knowledge that one has.

To build some intuition about these rates, one should also
consider the ``maximally timid strategy" where one chooses the first observation that falls in $[0, \, 0.5]$, then one chooses the next
observation that
falls in $[0.5, \, 1]$, and then
the next that falls in  $[0, \, 0.5]$, and so on. This strategy obviously leads to an asymptotic selection rate of $ 0.5 \, n$. Finally,
one should also consider the ``purely
greedy strategy" where one accepts any new arrival that is feasible given the previous selections. Curiously enough, by a reasonably quick
Markov chain calculation one can show that the greedy strategy leads to the same selection rate $ 0.5 \, n$
that one finds for the ``maximally timid strategy."

We begin by proving Theorem \ref{th2:geometric-mean} which will give us an exact formula for the expected number of selections made
under the optimal policy for geometric samples. This result will
then be used to prove the upper and lower bounds of Theorem \ref{th1:fixed-mean}.

\section{Proof of Theorem \ref{th2:geometric-mean}}\label{se:proof-thm-geometric-sample}

We now let $S_i$ denote the value of the last member of the subsequence
selected up to and including time $i$. To keep track of the up-down nature of our selections,
we then set $R_i=0$ if $S_i$ is a local minimum of
$\{S_0,S_1,\ldots,S_i\}$ and set $R_i=1$ if $S_i$ is a local
maximum. To initialize our process, we set $S_0 = 1$ and $R_0 = 1$.

Next, we make the class $\Pi$ of feasible policies more explicit.
For each $1 \leq i < \infty$
and for each pair $(S_{i-1}, R_{i-1})$ a feasible strategy  $\pi$ specifies a set $\Delta_{i}(S_{i-1}, R_{i-1})$
such that
$$
\Delta_{i}(S_{i-1}, 0) \subseteq [S_{i-1}, 1]
\quad
\text{ and }
\quad
\Delta_{i}(S_{i-1}, 1) \subseteq [0, S_{i-1}],
$$
and $X_i$  is selected for our subsequence if and only if
$
X_i \in \Delta_{i}(S_{i-1}, R_{i-1}).
$
For each $\pi \in \Pi$, we have the basic relation
\begin{equation*}
A^o_N(\pi) = \sum_{i=1}^N \1(X_i \in \Delta_{i}(S_{i-1}, R_{i-1}))
           = \sum_{i=1}^\infty \1(X_i \in \Delta_{i}(S_{i-1}, R_{i-1}))\1(i \leq N),
\end{equation*}
and by taking expectations on both sides we have
\begin{equation*}
\E[A^o_N(\pi)] = \E\left[ \sum_{i=1}^\infty \rho^{i-1} \1(X_i \in \Delta_{i}(S_{i-1}, R_{i-1}))\right].
\end{equation*}
We came to this relation by considering random sample sizes with the geometric distribution,
but the right side of this identity can also be
interpreted as the infinite-horizon discounted expected length
of the alternating subsequence selected by $\pi$. We are interested
in the policy $\pi^* \in \Pi$ such that
$$
\E[A^o_N(\pi^*)] = \sup_{\pi \in \Pi}
    \E\left[ \sum_{i=1}^\infty \rho^{i-1} \1(X_i \in \Delta_{i}(S_{i-1}, R_{i-1}))\right],
$$
and from the general theory of Markov decision problems,
we know that an optimal policy can be characterized as the solution of an associated Bellman equation.

\subsection*{First Bellman Equation}
For any $i$ such that $S_{i-1} = s$ and $R_{i-1}=r$ we let $v(s,r)$ denote the expected number of
selections made after time $i$ by an optimal policy. By the lack of memory property of the geometric distribution
and by the usual considerations of dynamic programming
one can now check that $v(s,r)$ satisfies Bellman equation:
\begin{equation}\label{eq:Bellman}
v(s,r)=
\begin{cases}
\, \rho s v(s,0) +\int_{s}^1 \max\left\{\rho v(s,0),1+ \rho v(x,1)\right\}\,dx      & \hbox{if $r=0$} \\
\, \rho (1-s)v(s,1) +\int_0^{s} \max\left\{ \rho v(s,1),1+ \rho v(x,0)\right\}\,dx  & \hbox{if $r=1$.}
\end{cases}
\end{equation}
To see why this equation holds, first consider the case when $r=0$ (so the next selection is to be a local maximum).
With probability $\rho$ we get to see another  observation $X_{i+1}$ and with probability $s$ the value we observe is less than the
previously selected value. In this case, we do not have the opportunity to make a selection,
and this observation contributes the term $\rho s v(s,0)$ to our equation.

Next, consider case when $s <X_{i+1} \leq 1$. Now one can choose to select $X_{i+1}=x$ or not.
If we do not select $X_{i+1}=x$  the expected number of subsequent selections is $\rho v(s,0)$ and if we do select  $X_{i+1}=x$
we increment sequence  by $1$ and the expected number of subsequence selections that are made by an optimal
policy in the future given by $\rho v(x,1)$.
Since $X_{i+1}$ is uniformly distributed in $[s,1]$ the expected optimal contribution
is given by the second term of our Bellman equation (top line).
The proof of the second line of the Bellman equation is completely
analogous.

Finally, given a solution $v(s,r)$ to the Bellman equation \eqref{eq:Bellman}, we have
$$
v(1,1) = \E[A^o_N(\pi^*)],
$$
so, now our goal is to determine $v(1,1)$. To do this it will be useful to reorganize the Bellman equation \eqref{eq:Bellman} in a tidier form.
This is possible since the solution $v(s,r)$ of the Bellman equation  has a useful symmetry property.

\begin{lemma}[Reflection Identity]\label{lm:symmetry}
For all $s\in [0,1]$ the  solution $v(s,r)$ of the Bellman equation \eqref{eq:Bellman} satisfies
\begin{equation}\label{eq:symmetry}
v(s,0) = v(1-s, 1).
\end{equation}
\end{lemma}

\begin{proof}
The Bellman equation \eqref{eq:Bellman} is a fixed point equation, and by the
classical theory of dynamic programming it can be solved by iteration
\citeaffixed[Sec. 9.5]{BerShr:AP1978}{c.f.}. We will prove the identity \eqref{eq:symmetry} by showing that it
holds for the sequence of approximations, so it also holds for the limit.

We first
set $v^0(s,r) = 0$ for all $(s,r) \in [0,1] \times \{0,1\}$
and we note that $v^0$ trivially satisfies the Reflection Identity \eqref{eq:symmetry}.
Next, for our induction hypothesis, we assume that we have $v^{n-1}(s,0) = v^{n-1}(1-s,1)$ for all $s\in [0,1]$.
The next iterate in the sequence is then given by
$$
v^{n}(s,0) = \rho s v^{n-1} (s,0) +\int_{s}^1 \max\left\{\rho v^{n-1}(s,0),1+ \rho v^{n-1}(x,1)\right\}\,dx.
$$
By applying our induction hypothesis on $v^{n-1}$ we then obtain
$$
v^{n}(s,0) = \rho s v^{n-1} (1-s,1) +\int_{s}^1 \max\left\{\rho v^{n-1}(1-s,1),1+ \rho v^{n-1}(1-x,0)\right\}\,dx.
$$
Now, after changing variables in the integral on the right-hand side, we find
\begin{align*}\label{iterateID}
v^{n}(s,0)  & = \rho s v^{n-1} (1-s,1) +\int_{0}^{1-s} \max\left\{\rho v^{n-1}(1-s,1),1+ \rho v^{n-1}(x,0)\right\}\,dx\\
            & = v^{n}(1-s,1), \notag
\end{align*}
and this completes the induction step. Now, for all $(s,r) \in [0,1]\times \{0,1\}$ we have $v^{n}(s,r) \rightarrow v(s,r)$
as $n\rightarrow \infty$ so taking limits in the last identity completes the proof of the reflection identity.\end{proof}

\subsection*{A Simpler Equation}

Using the reflection identity \eqref{eq:symmetry} we can put the Bellman equation \eqref{eq:Bellman} into a more graceful form.
Specifically, if we introduce a single variable function $v(y)$ defined by setting
$$
v(y)\equiv v(y,0) = v(1-y, 1),
$$
then substitution into our original equation \eqref{eq:Bellman} gives us
\begin{equation}\label{eq:Bellman-flipped}
v(y)  = \rho y v(y) +\int_{y}^1 \max\left\{\rho v(y),1+\rho v(1 - x)\right\}\,dx.
\end{equation}
Here we should note that by the definition
of $v(y)=v(y,0)$ we have that  $v(\cdot)$ is continuous,  $v(1)=0$, and $v$ is non-increasing on  $[0,1]$.
We will show shortly that $v$ is actually piecewise linear and it is constant on an initial segment of $[0,1]$.

\subsection*{An Alternative Interpretation}

The symmetrized equation \eqref{eq:Bellman-flipped} can be used to obtain a new probabilistic interpretation of $v(y)$.
To set this up, we first put
\begin{equation}\label{eq:f-star}
f^*(y) = \inf \{x \in [y, 1]:~\rho v(y) \leq 1+\rho v(1 - x) \}.
\end{equation}
With this definition, we can rewrite \eqref{eq:Bellman-flipped} as a bit more nicely as
\begin{equation}\label{eq:Bellman-flipped-2}
v(y)  = \rho f^*(y) v(y) +\int_{f^*(y)}^1 \left\{1+\rho v(1 - x)\right\}\,dx.
\end{equation}
Thus, one removes the maximum from the integrand \eqref{eq:Bellman-flipped} at the price of introducing
a threshold function $f^*$ that depends on $v$.

We now recursively define random variables $\{Y_i: i=1,2, \ldots\}$ by setting $Y_0=0$ and taking
$$
Y_i =
\begin{cases}
Y_{i-1}     &   \text{ if } X_i < f^*(Y_{i-1}) \\
1 - X_i     &   \text{ if } X_i \geq f^*(Y_{i-1}),
\end{cases}
$$
and finally introduce a new value function
\begin{equation}\label{eq:infinite-length-flipped}
v_0(y) \equiv \E\left[\sum_{i=1}^\infty \rho^{i-1} \1(X_i \geq f^*(Y_{i-1})) ~\bigg|~ Y_0 = y\right].
\end{equation}
The next proposition shows that $v_0(y)$ is actually equal to $v(y)$.
As part of the bargain, we obtain a concrete
characterization of the threshold function $f^*$.

\begin{proposition}[Structure of the Solution of the Bellman Equation]\label{StructureProp}
We have the following
characterizations of $f^*$ and $v_0$:
\begin{enumerate}[\rm (i)]
\item
There is a unique $\xi_0 \in [0,1]$ such that
\begin{equation*}\label{formOf}
f^*(y) = \max\{\xi_0, y\} \quad \text{for all } 0 \leq y \leq 1,
\end{equation*}
and moreover $0 \leq \xi_0 < 1/2$.

\item
The function $v_0(\cdot)$ is a solution of the Bellman equation \eqref{eq:Bellman-flipped}, so, by uniqueness, we have $v_0(y)=v(y)$
for all $0 \leq y \leq 1$.
\end{enumerate}
\end{proposition}

\begin{proof}
From the definition of $f^*$ we see  that
\begin{equation} \label{eq:f-star-equal-y}
\rho v( y ) \leq  1 + \rho v( 1-y ) \quad \Rightarrow \quad f^*(y) = y.
\end{equation}
Now, for $1/2 \leq y$ we have $1-y \leq y$, so
the monotonicity of $v$ gives us the bound $\rho v( y ) \leq  1 + \rho v( 1-y )$; consequently, we have  $f^*(y)=y$ for $y \in [1/2,1]$.

If the condition \eqref{eq:f-star-equal-y} holds for all $y\in [0,1/2)$, then $f^*(y)=y$ for all $y\in [0,1]$ and we can take $\xi_0=0$.
Otherwise there is a $y_0\in [0,1/2)$ for which we have
\begin{equation*}
1 + \rho v( 1-y_0 ) <\rho v( y_0 ).
\end{equation*}
For $\Delta(y)= 1 + \rho v( 1-y ) - \rho v( y )$ we then have $\Delta(y_0)<0$ and $\Delta(1)= 1+\rho v(0) >0$, so by continuity
we have $S=\{y: \Delta(y)=0 \} \not= \emptyset$. If we now take $\xi_0$ to be the infimum of $S$, then
$\xi_0 \in [y_0, 1/2) \subset [0,1/2)$
and
$
\rho v( \xi_0 ) =  1 + \rho v( 1-\xi_0 ).
$
The definition of $f^*$ now tells us that  $f^*(y)=\xi_0$ for $y\leq \xi_0$ and
 $f^*(y)=y$ for $\xi_0 \leq y.$ This completes the proof of the first part of the proposition.

Finally, to check that $v_0$ solves the equation \eqref{eq:infinite-length-flipped}, we just condition on the value of $X_1$
and calculate the expectation of
the sum.  When we take the total expectation, we get the right side of \eqref{eq:Bellman-flipped-2}.
\end{proof}

\subsection*{Characterization of the Critical Value}

Now that we know that the threshold function $f^*$ for the solution of
Bellman equation \eqref{eq:Bellman-flipped} has the form  $f^*(y) = \max\{\xi_0,y\}$ for some $\xi_0 \in [0,1/2)$,
the main problem is to find $\xi_0$.
The natural plan is to fix $\xi \in [0,1/2]$ and to consider a general selection function of the form
$f(y) = \max\{\xi,y\} \equiv (\xi \vee y)$. We then want to calculate the associated value function and to optimize over $\xi$.

The associated value function is given by
\begin{equation}\label{eq:V-sum}
V(y,\xi, \rho) = \E\left[\sum_{i=1}^\infty \rho^{i-1} \1(X_i \geq \max\{\xi, Y_{i-1}\}) ~\bigg|~ Y_0 = y\right],
\end{equation}
and Proposition \ref{StructureProp} then tells us that
$$
\max_{\xi\in [0, 1/2]}V(y,\xi, \rho) = v(y) \quad \text{ for all } y\in [0,1].
$$
If we abbreviate $ V(y,\xi, \rho)$ by setting
$V(y) \equiv V(y,\xi, \rho)$, then by conditioning on $X_1$ in equation \eqref{eq:V-sum} we see that $V(y)$
satisfies the integral equation
\begin{align}
V(y)  & = (\xi\vee y) \rho V(y) + \int_{\xi\vee y}^1 \{1 + \rho V(1 - x) \} \, dx \nonumber\\
      & = (\xi\vee y) \rho V(y) + \int_0^{1-(\xi\vee y)} \{1 + \rho V(x)\} \, dx. \label{eq:V}
\end{align}
This equation has several attractive features.
In particular, if we set $y=1$ then from $0< \rho <1$ we see $V(1) = 0$. Also, by writing
\begin{equation*}
V(y)= \frac{1}{1- \rho (\xi\vee y)}\int_0^{1-(\xi\vee y)} \{1 + \rho V(x) \} \, dx,
\end{equation*}
we see that the right side does not change when $y \in [0, \xi]$, so we have
\begin{equation}\label{eq:V-constant}
V(y) = V(y') \quad\text{for all $0\leq y, \, y' \leq \xi$}.
\end{equation}

From now on, we will let  $V'(\xi)$ denote the right derivative of the
integral equation \eqref{eq:V} evaluated at $\xi$, and let $V'(1-\xi)$ denote
the left derivative of \eqref{eq:V} evaluated at $1-\xi$. Elsewhere $V'(y)$
simply denotes the derivative of \eqref{eq:V} evaluated at $y$.

\begin{lemma}\label{lm:4conditions}
The solution of equation \eqref{eq:V} satisfies the following four conditions:
\begin{enumerate}[\rm (i)]
  \item\label{condition1}
            $V(1-\xi) ( 1 - \rho + \rho \xi ) = \xi + \rho \xi V(\xi)$;
  \item\label{condition2}
            $V'(\xi) (1 - \rho \xi) = \rho [ V(\xi) - V(1-\xi) ] -1$;
  \item\label{condition3}
            $V'(1 - \xi) (1 - \rho + \rho \xi) = \rho [ V(1 - \xi) - V(\xi) ] -1$;
  \item\label{condition4}
            $V'(1 - \xi) (1 - \rho + \rho \xi)^2 (1 - \rho \xi) = V'(\xi)(1 - \rho \xi)^2 (1 - \rho + \rho \xi) + (1 - \rho + \rho \xi)^2 - (1 - \rho \xi)^2$.
\end{enumerate}
\end{lemma}

\begin{proof}
Conditions \eqref{condition1}--\eqref{condition3}
are easy to check. Condition \eqref{condition1} is just \eqref{eq:V}
evaluated at $1 - \xi$ together with \eqref{eq:V-constant}.
Conditions \eqref{condition2} and \eqref{condition3} simply
follow by evaluating \eqref{eq:V} at $\xi$ and $1-\xi$ respectively
and by differentiating both sides with respect to $\xi$.

The proof of Condition \eqref{condition4} requires more work.
Consider $ y \in (\xi,  1-\xi)$  so that the integral equation \eqref{eq:V} becomes
$$
V(y) = y \rho V(y) + \int_0^{1- y} \{1 + \rho V(x) \} \, dx.
$$
Differentiating once we have
\begin{equation}\label{eq:V'(y)}
V'(y) (1-\rho y) = \rho [ V(y) - V(1-y) ] - 1,
\end{equation}
and differentiating again gives us
\begin{equation}\label{eq:V''(y)}
V''(y)(1-\rho y) - \rho V'(y) = \rho V'(y) + \rho V'(1-y).
\end{equation}
To estimate the value of $V'(1-y)$ we note that $1-y \in (\xi,  1-\xi)$,
and we evaluate the integral equation \eqref{eq:V} at $1 - y$.
We then differentiate with respect to $y$ to obtain
\begin{equation}\label{eq:V'(1-y)}
V'(1 - y) (1 - \rho + \rho y) = \rho [ V(1 - y) - V(y) ] -1.
\end{equation}
By combining equations \eqref{eq:V'(y)} and \eqref{eq:V'(1-y)} we then have
\begin{equation*}\label{eq:V'(1-y)-value}
V'(1 - y) = (1 - \rho + \rho y)^{-1} ( - V'(y) (1-\rho y) - 2 ),
\end{equation*}
which we can plug into the last addend of \eqref{eq:V''(y)} to obtain
\begin{equation}\label{eq:V''(y)-bis}
V''(y)(1-\rho y)(1 - \rho + \rho y) = V'(y) \rho (1 - 2 \rho + 3 \rho y) - 2 \rho.
\end{equation}
By multiplying both sides of \eqref{eq:V''(y)-bis} by $(1-\rho y)$ we obtain the critical identity
\begin{equation}\label{eq:V''(y)-fundamental}
V''(y)(1-\rho y)^2(1 - \rho + \rho y) = V'(y) \rho (1 - \rho y) (1 - 2 \rho + 3 \rho y) - 2 \rho (1-\rho y).
\end{equation}
For
$
h(y) = (1-\rho y)^2(1 - \rho + \rho y)
$
notice that
$
h'(y) = - \rho (1 - \rho y) (1 - 2 \rho + 3 \rho y),
$
so that we can rewrite the identity \eqref{eq:V''(y)-fundamental} as
$$
V''(y)h(y) + V'(y) h'(y) - \left[(1 - \rho y)^2\right]' = 0.
$$
An immediate integration then gives us
$$
V'(y) h (y) - (1 - \rho y)^2 = C,
$$
where $C$ is a constant, and if we take
$C = V'(\xi) h (\xi) - (1 - \rho \xi)^2$ we find
\begin{equation}\label{eq:V'(y)-bis}
V'(y) = V'(\xi) \frac{h (\xi)}{h (y) } +  \frac{(1 - \rho y)^2 - (1 - \rho \xi)^2 }{h(y)}\quad \text{for all $\xi< y < 1-\xi$.}
\end{equation}
Finally, on setting $y=1-\xi$ we recover the desired condition \eqref{condition4}.
\end{proof}

\subsection*{Calculation of the Critical Value.}

Conditions \eqref{condition1}--\eqref{condition4} in Lemma \ref{lm:4conditions}
generate a system of four equations in four unknowns,
$V(\xi), V(1-\xi), V'(\xi)$, and  $V'(1-\xi)$. By solving this system one finds
\begin{eqnarray}
  V(\xi) &=& \frac{2 - 2\xi - \rho + 2\rho\xi - 2\rho\xi^2}{2(1-\rho)(1-\rho \xi)} \label{eq:V(xi)-explicit}\\
  V(1-\xi) &=& \frac{\rho\left(2 - 4 \rho \xi - \rho^2 + 4\rho^2 \xi - 2 \rho^2 \xi^2 \right)}{2(1-\rho) (1-\rho \xi)^2  (1-\rho + \rho\xi)} \nonumber\\
  V'(\xi) &=& \frac{-2 + 4\rho - 4\rho\xi -\rho^2 + 2\rho^2 \xi^2}{2(1-\rho \xi)^2  (1-\rho + \rho\xi)} \label{eq:V'(xi)-explicit}\\
  V'(1-\xi) &=& \frac{-2 + 4\rho \xi +\rho^2 - 4\rho^2 \xi + 2 \rho^2 \xi^2}{2(1-\rho \xi) (1-\rho + \rho\xi)^2}.\nonumber
\end{eqnarray}
Finally, by substituting \eqref{eq:V'(xi)-explicit} into \eqref{eq:V'(y)-bis} we get
\begin{equation*}
V'(y) = \frac{-(2-\rho)^2 + 2(1 - \rho y)^2}{2(1-\rho + \rho y)(1-\rho y)^2} \quad \text{for all $\xi< y < 1-\xi$.}
\end{equation*}

Now, given any $\xi$, we want to compute $V(0,\xi, \rho)$.
We first recall that we have $V(1,\xi, \rho) = 0$ and
$V(y,\xi, \rho) = V(\xi,\xi, \rho)$ for all $0\leq y \leq \xi$. We therefore find that
$\frac{\partial}{\partial y}V(y,\xi, \rho)=0$ on $ 0 \leq y \leq \xi$, so on integrating we have
$$
V(1,\xi, \rho) - V(0,\xi, \rho) = \int_0^1 V'(y) \, dy = \int_\xi^1 V'(y) \, dy
$$
and hence
$$
 V(0,\xi, \rho) = - \int_\xi^1 V'(y) \, dy.
$$
We now optimize this last quantity with respect to $\xi$.
By differentiating both sides with respect to $\xi$ we get
$$
\frac{\partial}{\partial \xi} V(0,\xi, \rho) = V'(\xi)
$$
and we are interested in the value $\xi_0$ such that
$$
V'(\xi_0) = 0.
$$
Our formula \eqref{eq:V'(xi)-explicit} for $V'(\xi_0)$ tells us
that $V'(\xi_0) = 0$ if and only if
$$
 2(1 - \rho \xi_0)^2 = (2-\rho)^2.
$$
We therefore find that the unique choice for $\xi_0$ is given by
\begin{equation}\label{eq:xi0}
\xi_0 = \frac{1}{\sqrt{2}} + \frac{1-\sqrt{2}}{\rho}.
\end{equation}
A routine calculation verifies that $V''(\xi_0) < 0$,
so we have found our maximum.

When we evaluate $V(\xi_0,\xi_0, \rho)$ using
equation \eqref{eq:V(xi)-explicit}, we find
$$
V(\xi_0,\xi_0, \rho) = \frac{3 - 2\sqrt{2} - \rho + \rho\sqrt{2}}{\rho (1 - \rho)},
$$
and this gives us the main formula of  Theorem \ref{th2:geometric-mean}.
From this formula it is immediate that
$$
\lim_{\rho \uparrow 1} (1 - \rho) V(\xi_0, \xi_0, \rho) = 2 - \sqrt{2},
$$
so the proof of Theorem \ref{th2:geometric-mean} is complete.

\section{Proof of Theorem \ref{th1:fixed-mean} from Theorem \ref{th2:geometric-mean}}\label{se:proof-thm-finite-sample}

We will use our results for geometric sample sizes to get both lower and upper bounds for the
finite sample size selection problem. The lower bound is the easiest. For fixed $n$, one can use
the (now suboptimal) policy from an appropriately chosen geometric sample size problem. The proof of the
upper bound is considerably harder, and the method will be described later in this section.
Before making these arguments, we need to organize a few structural observations.

\subsection*{Selection Policies and a Bellman Equation for Finite Samples}

When the sample size $n$ is deterministic and known, the feasible policies need to take this information into account.
In particular, the selection thresholds will no longer be stationary;
they will depend on the number of sample elements that remain to be seen.

Just as in Section \ref{se:proof-thm-geometric-sample}, we consider the pairs  $(S_{i-1}, R_{i-1})$, $1\leq i \leq n$, where $S_{i-1}$ is the size of the
last selection made before time $i$ and $R_{i-1}$ is $0$ or $1$ accordingly as the last selection was a
local minimum or a local maximum.
A feasible policy $\pi \in \Pi$ again specifies a set $\Delta_{i,n}(S_{i-1}, R_{i-1})$
that depends only on past actions, but now we have dependence on the decision time $i$
and on the sample size $n$.
For any policy   $\pi \in \Pi$ the expected size of the selected sample can then be written as
\begin{equation*}
\E[A^o_n(\pi)] = \E\left[ \sum_{i=1}^n \1(X_i \in \Delta_{i,n}(S_{i-1}, R_{i-1}))\right]
\end{equation*}
and there is an optimal policy $\pi^*_n$ for which we have
$$
\E[A^o_n(\pi^*_n)] = \sup_{\pi \in \Pi} \E[A^o_n(\pi)].
$$
In this case, an optimal policy can be characterized as the solution to a finite sample Bellman equation.
Specifically, for $1\leq i \leq n$, we have
\begin{equation*}
v_{i,n}(s,r)=
\!\!
\left\{
 \begin{array}{ll}
\!\!\! sv_{i+1,n}(s,0) +\int_{s}^1 \max\left\{v_{i+1,n}(s,0),1+v_{i+1,n}(x,1)\right\}\,dx
 & \hbox{if $r=0$} \\
\!\!\! (1-s)v_{i+1,n}(s,1) +\int_0^{s} \max\left\{v_{i+1,n}(s,1),1+v_{i+1,n}(x,0)\right\}\,dx
 & \hbox{if $r=1$,}\\
 \end{array}
 \right.
\end{equation*}
and the backward induction begins by setting  $v_{n+1,n}(s,r) \equiv 0$ for all $(s,r)$ in $[0,1]\times \{0,1\}$.
This equation is justified by the same considerations that were used in the derivation of equation \eqref{eq:Bellman}.

\subsection*{Symmetry and Simplification}

For the finite sample size problem, one loses much of the nice symmetry of the geometric sample size problem. Nevertheless,
the solution of the finite sample Bellman equation still has a
reflection identity analogous to that given by Lemma \ref{lm:symmetry}.

\begin{lemma}\label{lm:symmetryFS}
The solution of the finite sample Bellman equation
satisfies
\begin{equation}\label{eq:reflection2}
v_{i,n}(s,0) = v_{i,n}(1-s,1)\quad \text{ for all }1\leq i \leq n \text{ and all } s \in [0,1].
\end{equation}
\end{lemma}

\begin{proof}
Again we use an induction argument, but this time we do not need to take limits of an infinite sequence of approximate solutions.
Instead we simply use backward induction and always work with exact solutions.

Since we have $v_{n,n}(s,0) = 1-s$ and $v_{n,n}(1-s,1)=1-s$, we see that
equation \eqref{eq:reflection2}
holds for $i=n$, so we
suppose by induction that  $v_{i+1,n}(s,0) = v_{i+1,n}(1-s,1)$.
One then has
$$
v_{i,n}(s,0)= sv_{i+1,n}(s,0) +\int_{s}^1 \max\left\{v_{i+1,n}(s,0),1+v_{i+1,n}(x,1)\right\} \, dx,
$$
so by applying the induction hypothesis on the right-hand side one obtains
$$
v_{i,n}(s,0)= sv_{i+1,n}(1-s,1) +\int_{s}^1 \max\left\{v_{i+1,n}(1-s,1),1+v_{i+1,n}(1-x,0)\right\} \, dx.
$$
If we now change variable in this last integral, we get
\begin{align*}
v_{i,n}(s,0) & = sv_{i+1,n}(1-s,1) +\int_{0}^{1-s} \max\left\{v_{i+1,n}(1-s,1),1+v_{i+1,n}(x,0)\right\} \, dx\\
             & = v_{i,n}(1-s,1),
\end{align*}
and this completes the induction step.
\end{proof}

We can now define a new single variable function $v_{i,n}(y)$ by
setting
\begin{equation}\label{eq:Bellman-finite-equivalence}
v_{i,n}(y) = v_{i,n}(y,0) = v_{i,n}(1-y,1)
\end{equation}
and, by substitution into the original finite sample Bellman equation
we have
\begin{equation}\label{eq:Bellman-FINITE-flipped}
v_{i,n}(y)  = y v_{i+1,n}(y) +\int_{y}^1 \max\left\{ v_{i+1,n}(y),1+ v_{i+1,n}(1 - x)\right\}\,dx.
\end{equation}
Here we should also note that $v_{i,n}(\cdot)$ is continuous and non-increasing on $[0,1]$
for all $1\leq i \leq n$.

\subsection*{The Threshold Functions}

We now define the finite-sample equivalent of the threshold function \eqref{eq:f-star} by setting
\begin{equation}\label{eq:f-FINITE-star}
f^*_{i,n}(y) = \inf \{x \in [y, 1]:~ v_{i+1,n}(y) \leq 1+ v_{i+1,n}(1 - x) \}.
\end{equation}
If we then set $Y_0 = 0$ and  define $Y_i$ recursively by setting
\begin{equation}\label{eq:Yi-optimal-finite}
Y_i =
\begin{cases}
    Y_{i-1} & \text{if $X_i< f^*_{i,n}(Y_{i-1})$}\\
    1 - X_{i} & \text{if $X_i\geq f^*_{i,n}(Y_{i-1})$},
\end{cases}
\end{equation}
then, in complete parallel to the geometric case, we see that the solution of the finite
sample Bellman equation \eqref{eq:Bellman-FINITE-flipped}
can be written more probabilistically as
\begin{equation}\label{eq:v1n}
v_{1,n}(y) = \E\left[\sum_{i=1}^n \1(X_i \geq f^*_{i,n}(Y_{i-1})) ~\bigg|~ Y_0 = y\right].
\end{equation}
Finally, from equation \eqref{eq:Bellman-finite-equivalence} we have
$$
v_{1,n}(0) = v_{1,n}(0,0) = v_{1,n}(1,1) = \E[A^o_n(\pi^*_n)],
$$
and this gives us the last piece of structural information that we need.

\subsection*{Proof of the Lower Bound}
To prove that
$$
(2 - \sqrt{2}) n \leq \E[A^o_n(\pi^*_n)] \quad \text{ for all } n\geq 1
$$
we only need to choose a good suboptimal policy. We now
fix $\xi\in [0,1/2]$ and we consider the policy in which
$X_i$ is selected if and only if $X_i \geq \max\{\xi, Y_{i-1}\}$.
Here, $Y_0 = y$ is in the interval  $[0, 1-\xi]$ and the $Y_i$'s are defined recursively by setting
$$
Y_i =
\begin{cases}
    Y_{i-1} & \text{if $X_i< \max\{\xi, Y_{i-1}\}$}\\
    1 - X_{i} & \text{if $X_i\geq \max\{\xi, Y_{i-1}\}$}.
\end{cases}
$$
The sequence $\{Y_i: i=0,1,\ldots\}$  is a discrete-time Markov Chain on the state space $[0, 1-\xi]$.
For a measurable $A\subseteq [0, 1-\xi]$ we let $|A|$ denote the Lebesgue measure of $A$, and for a measurable set $B \subseteq [0, 1-\xi]$
we write $1-B$ as shorthand for the set $\{u \in [0,1]:~ 1-u \in B\}$. Given these abbreviations, the
transition kernel of the process $\{Y_i: i=0,1,\ldots\}$  can
be written as
$$
K(x,C) = \1(x\in C) (\xi \vee x) + |1 - C \cap [\xi \vee x, 1]|.
$$
It is now easy to check that the process $\{Y_i\}$ has a unique stationary distribution $\gamma$,
and in fact  $\gamma$ is just the uniform distribution on $[0, 1-\xi]$,
(i.e., $\gamma(C) = (1-\xi)^{-1}|C|$ for all measurable $C\subseteq [0, 1-\xi]$).

For any starting value $Y_0=y \in [0, 1-\xi]$, the suboptimality
of the selection functions $\max\{\xi, Y_{i-1}\}$
gives that
$$
\E\left[\sum_{i=1}^n \1(X_i \geq \max\{\xi, Y_{i-1}\}) ~\bigg|~ Y_0 = y\right] \leq v_{1,n}(y).
$$
Since $v_{1,n}(y)$ is non-increasing in $y$, we see that for any starting distribution $\mu$ supported on $[0, 1-\xi]$
one has
$$
\E_\mu\left[\sum_{i=1}^n \1(X_i \geq \max\{\xi, Y_{i-1}\})\right] \leq \E_\mu[v_{1,n}(Y_0)] \leq v_{1,n}(0) =  \E[A^o_n(\pi^*_n)].
$$
If one chooses the starting distribution $\mu$ to be the stationary distribution $\gamma$,
then
\begin{equation}\label{eq:stat-dist-LB}
\E_\gamma\left[\sum_{i=1}^n \1(X_i \geq \max\{\xi, Y_{i-1}\})\right]
= n\, \E_\gamma\left[1- \max\{\xi, Y_{0}\})\right]
\leq \E[A^o_n(\pi^*_n)],
\end{equation}
and we can compute the first expression explicitly. So, we have
$$
\E_\gamma\left[1- \max\{\xi, Y_{0}\}\right]
= \frac{1}{1-\xi} \int_0^{1-\xi} 1- \max\{\xi, y\}\,dy
= \frac{1 - 2\xi^2}{2(1-\xi)}.
$$
We can maximize this by taking $\xi = 1 - 2^{-1/2}$ (as in \eqref{eq:xi0} when $\rho = 1$),
and we then obtain
$$
\E_\gamma\left[1- \max\{\xi, Y_{0}\})\right]  = 2 - \sqrt{2}.
$$
Together with the inequality \eqref{eq:stat-dist-LB}, this completes the proof of our lower bound.

\subsection*{Proof of the Upper Bound}

The proof of the upper bound in Theorem \ref{th1:fixed-mean} requires a more sustained argument.
Unlike the problem for geometric samples, the value function  $v_{i,n}(\cdot)$
is no longer constant on an initial segment of $[0,1]$. Nevertheless, the next proposition tells us that
the value function does have a useful uniform boundedness on an initial segment.
This is the first of several structural observations that we will need to obtain our
upper bound for $\E[A^o_n (\pi^*_n)]$.

\begin{proposition}[Value Function Initial Segment Bounds]\label{pr:f-lower-bound}
For all $0 \leq u <1/6$ and $n\geq 2$, the functions $v_{i,n}(\cdot)$
defined by the Bellman recursion \eqref{eq:Bellman-FINITE-flipped}
satisfy
\begin{enumerate}[\rm (i)]
    \item \label{it:LB-Bellman-u-5/6}
            $ 1 < v_{i,n}(u) - v_{i,n}(5/6)$,  for all  $ 1 \leq i \leq n - 1$;
    \item \label{it:UB-Bellman-u-1/6}
            $ v_{i,n}(u) - v_{i,n}(1/6) < 1$, for all $ 1 \leq i \leq n $.
\end{enumerate}
Moreover, for $n\geq 3$, the threshold functions $f^*_{i,n}(y)$ defined by equation \eqref{eq:f-FINITE-star}
are guaranteed to satisfy $1/6 \leq f^*_{i,n}(y)$ for $y \in [0,1]$ and $1 \leq i \leq n - 2$.
\end{proposition}

Naturally enough, the proof of this proposition depends on inductive arguments that exploit the defining Bellman equation.
The first of these arguments gives us some control over the
changes of $v_{i,n}(u)$ when we change both $i$ and $u$.

\begin{lemma}[Restricted Supermodularity]\label{lm:Bellman-growth}
For $y\in [0,1/2]$ and $u\in [y, 1-y]$ the functions $\{ v_{i,n}(\cdot)\}$ defined by the  Bellman recursion
\eqref{eq:Bellman-FINITE-flipped} satisfy
$$
v_{i+1,n}(u) - v_{i+1,n}(1-y) \leq v_{i,n}(u) - v_{i,n}(1-y) \quad \text{for all $1\leq i\leq n$.}
$$
\end{lemma}

\begin{proof}
We use backward induction on $i$, and, since $n$ is fixed, we abbreviate $ v_{i,n}(\cdot)$ by
$v_i(\cdot)$.
For $i=n$ we have $v_{n+1}(u) = 0$ for all $u\in [0,1]$. Moreover,
$v_n(u) = 1 - u$ and $v_n(1-y) = y$, so we have
\begin{equation*}\label{eq:induction-assumptionBis}
v_{n+1}(u) - v_{n+1}(1-y) \leq v_{n}(u) - v_{n}(1-y) \quad \text{for all $u \in [y, 1-y]$}.
\end{equation*}
Now, for our backward induction, we can assume more generally that
\begin{equation*}
v_{i+1}(u) - v_{i+1}(1-y) \leq v_{i}(u) - v_{i}(1-y) \quad \text{for all $u \in [y, 1-y]$}.
\end{equation*}
The Bellman equation \eqref{eq:Bellman-FINITE-flipped} then gives us
\begin{align*}
v_{i-1}(u) - v_{i-1}(1-y) & = u v_{i}(u) +\int_{u}^1 \!\!\! \max\left\{ v_{i}(u),1+ v_{i}(1 - x)\right\}\,dx\\
                          & - (1-y) v_{i}(1-y) -\int_{1-y}^1 \!\!\!\! \max\left\{ v_{i}(1-y),1+ v_{i}(1 - x)\right\}\,dx,
\end{align*}
and, since $u\leq 1-y$, we can break up the first integral to obtain
\begin{align}
v_{i-1}(u) -& v_{i-1}(1-y)  = u v_{i}(u) \! - \! (1-y) v_{i}(1-y) + \! \int_{u}^{1-y} \!\!\!\!\!\!\! \max\left\{ v_{i}(u),1+ v_{i}(1 - x)\right\}\,dx \nonumber\\
                          &  + \int_{1-y}^1 \!\!\!\! \max\left\{ v_{i}(u),1+ v_{i}(1 - x)\right\} - \max\left\{ v_{i}(1-y),1+ v_{i}(1 - x)\right\}\,dx. \label{eq:integral-max-diff}
\end{align}
For $x\in [1-y, 1]$, we have
$v_{i}(y) \leq v_i(1-x)$ since $v_i(\cdot)$ is non-increasing on $[0,1]$.
Therefore, since $y \leq u \leq 1-y$ we have $v_{i}(1 - y) \leq v_{i}(u) \leq v_{i}(y)$
so that for $x\in [1-y, 1]$ we have
$$
\max\left\{ v_{i}(u),1+ v_{i}(1 - x)\right\} = \max\left\{ v_{i}(1-y),1+ v_{i}(1 - x)\right\} = 1+ v_{i}(1 - x),
$$
and we see that the integral \eqref{eq:integral-max-diff} equals $0$.
We now have just the identity
$$
v_{i-1}(u) - v_{i-1}(1-y)  = u v_{i}(u) - (1-y) v_{i}(1-y) +\int_{u}^{1-y} \!\!\!\!\! \max\left\{ v_{i}(u),1+ v_{i}(1 - x)\right\}\,dx
$$
or, equivalently,
\begin{align*}
v_{i-1}(u) - v_{i-1}(1-y) & = u \left(v_{i}(u) - v_{i}(1-y)\right) \\
                          & + \int_{u}^{1-y} \!\!\!\!\! \max\left\{ v_{i}(u) - v_{i}(1-y), 1+ v_{i}(1 - x) - v_{i}(1-y)\right\}\,dx.
\end{align*}
Changing variables in this last integral then gives us the convenient identity
\begin{align}\label{eq:recursion-Bellman-difference}
v_{i-1}(u) - v_{i-1}(1-y) & = u \left(v_{i}(u) - v_{i}(1-y)\right) \nonumber\\
                          & + \int_{y}^{1-u}  \!\!\!\!\! \max\left\{ v_{i}(u) - v_{i}(1-y), 1+ v_{i}( x ) - v_{i}(1-y)\right\}\,dx.
\end{align}
Since $ y \leq u $ and $ 1-u \leq 1-y $, we can now use our induction
assumption to obtain
\begin{align*}
v_{i-1}(u) - v_{i-1}(1-y) & \geq u \left(v_{i+1}(u) - v_{i+1}(1-y)\right) \\
                          & + \! \int_{y}^{1-u} \!\!\!\!\!\!\!\!\!\! \max\left\{ v_{i+1}(u) \! -  \! v_{i+1}(1-y), 1 \! + \! v_{i+1}( x ) \! - \! v_{i+1}(1-y)\right\}\,dx\\
                          & =  v_{i}(u) - v_{i}(1-y),
\end{align*}
where the last equality follows from the recursion \eqref{eq:recursion-Bellman-difference}.
\end{proof}

We can now complete the proof of the  Value Function Bounds in Proposition \ref{pr:f-lower-bound}.

\begin{proof}[Proof of Proposition \ref{pr:f-lower-bound}]
We begin by proving \eqref{it:LB-Bellman-u-5/6} by backwards induction on $i$. As before,
since $n \geq 2$ is fixed, we abbreviate $v_{i,n} (\cdot)$ by $v_i (\cdot)$.
For $i=n-1$ one iteration of the recursive definition of the Bellman
equation \eqref{eq:Bellman-FINITE-flipped} gives us that
$v_{n-1}(x) = (3/2) (1-x^2)$, so
$v_{n-1}(u) - v_{n-1}(5/6) = (3/2)(25/36 - u^2) > 1$ since by hypothesis we have $u < 1/6$.
We now make the induction assumption
\begin{equation*}
1 < v_{i+1}(u) - v_{i+1}(5/6) \quad \text{ for $0 \leq u <1/6$},
\end{equation*}
and observe from the Bellman equation \eqref{eq:Bellman-FINITE-flipped} that
\begin{align*}
 v_{i }(u)  - v_{i }(5/6)
       & = u  v_{i + 1 }(u) + \int_u^1\max\{v_{i + 1 }(u), 1 + v_{i + 1 }(1 - x) \} \, dx\\
       & - 5/6 \, v_{i + 1 }(5/6) - \int_{5/6}^1\max\{v_{i + 1 }(5/6), 1 + v_{i + 1 }(1 - x) \} \, dx.
\end{align*}
Since $u < 5/6$, the monotonicity of $v_i(\cdot)$ implies $v_{i + 1 }(5/6) \leq v_{i + 1 }(u)$. So,
for  $x \in [5/6, 1]$, we have
$
\max\{v_{i + 1 }(5/6), 1 + v_{i + 1 }(1 - x) \}
\leq
\max\{v_{i + 1 }(u), 1 + v_{i + 1 }(1 - x) \}.
$
This gives us the lower bound
\begin{align*}
u \left(  v_{i + 1 }(u) \! - \! v_{i + 1 }(5/6) \right)
&  \! + \!  \int_u^{5/6} \!\!\!\!\!\! \max\{v_{i + 1 }(u)  \! - \! v_{i + 1 }(5/6),  1  \! +  \! v_{i + 1 }(1 - x)  \! - \! v_{i + 1 }(5/6)\} \, dx\\
& \leq
v_{i }(u)  - v_{i }(5/6).
\end{align*}
To get a lower bound for the integral of the maximum,  we replace the integrand by $v_{i + 1 }(u)  \! - \! v_{i + 1 }(5/6)$
on $[u, 1/6)$ and replace it by $1  \! +  \! v_{i + 1 }(1 - x)  \! - \! v_{i + 1 }(5/6)$ on $[1/6, 5/6]$.
Changing variables then gives us
\begin{equation}\label{eq:inequality-two-integrals}
\frac{1}{6} \left( v_{i + 1 }(u) - v_{i + 1 }(5/6) \right)
+ \int_{1/6}^{5/6}  \!\!\!\!\! \{ 1 + v_{i + 1 }( x ) -v_{i + 1 }(5/6)\}  \, dx
  \leq
v_{i }( u )  - v_{i }(5/6).
\end{equation}
By our induction hypothesis the first addend
satisfies the bound
\begin{equation}\label{eq:lower-bound1/6}
\frac{1}{6} < \frac{1}{6} \left( v_{i + 1 }(u) - v_{i + 1 }(5/6) \right),
\end{equation}
and by Lemma \ref{lm:Bellman-growth},
the second integral satisfies the bound
$$
\int_{1/6}^{5/6} \{1 + v_{n}(x) -v_{n}(5/6)\} \, dx \leq \int_{1/6}^{5/6} \{1 + v_{i + 1 }(x) -v_{i + 1 }(5/6)\} \, dx.
$$
If we now recall that $v_{n}(x) = 1-x$ and
compute the integral on the left-hand side we then obtain
\begin{equation}\label{eq:lower-bound32/36}
\frac{32}{36} \leq \int_{1/6}^{5/6} \{ 1 + v_{i + 1 }(x) -v_{i + 1 }(5/6) \} \, dx .
\end{equation}
Finally, adding \eqref{eq:lower-bound1/6} and \eqref{eq:lower-bound32/36} and
recalling \eqref{eq:inequality-two-integrals} gives us our target bound
$$
1 < \frac{38}{36} \leq v_{i }(u)  - v_{i }(5/6).
$$

\bigskip

To prove condition \eqref{it:UB-Bellman-u-1/6} we again use
backwards induction.
For $i=n$ we have $v_n(u) = 1 - u$, so $v_n(u) - v_n(1/6) = 1/6 - u < 1$.
Suppose now that
$$
v_{i+1} ( u ) - v_{i+1} (1/6) < 1 \quad \text{ for $0 \leq u <1/6$}.
$$
The Bellman recursion \eqref{eq:Bellman-FINITE-flipped}
then gives us
\begin{align*}
v_{i} ( u ) & - v_{i} (1/6)  \leq
              \int_0^{1/6} \!\!\!\!\!\! \max\{v_{i + 1 }(u)  \! - \! v_{i + 1 }(1/6),  1  \! +  \! v_{i + 1 }(1 - x)  \! - \! v_{i + 1 }(1/6)\} \, dx \\
            & + \int_{1/6}^{5/6} \!\!\!\!\!\! \max\{v_{i + 1 }(u) ,  1  +  v_{i + 1 }( x) \} - \max\{v_{i + 1 }(1/6) ,  1  +  v_{i + 1 }( x) \}  \, dx \\
            & + \int_{5/6}^{1} \!\!\!\!\!\! \max\{v_{i + 1 }(u) ,  1  +  v_{i + 1 }(1 - x) \} - \max\{v_{i + 1 }(1/6) ,  1  +  v_{i + 1 }(1 - x) \} \, dx.
\end{align*}

For $x \in [0, 1/6]$, we can check that first integrand is bounded by 1. To see this, we first note that left
maximand is bounded by 1 by the induction assumption. Next, we note that $v_{i+1}(1-x) \leq v_{i+1}(5/6)$
so, for the second maximand one has the bound
$
1  +  v_{i + 1 }(1 - x)  - v_{i + 1 }(1/6)
\leq 1  +  v_{i + 1 }(5/6)  - v_{i + 1 }(1/6)
$
and this last term is non-positive by the inequality \eqref{it:LB-Bellman-u-5/6}.

For $x \in [1/6, 5/6]$ the second integrand is bounded by
$$
 \max\{v_{i + 1 }(u) - v_{i + 1 }(1/6) ,  1  +  v_{i + 1 }( x)  - v_{i + 1 }(1/6)  \}
 \leq 1,
$$
since both maximands are bounded by 1; the first one because of the induction assumption
and the second one because it is non-increasing in $x$ and attains its maximum for $x = 1/6$.

Finally, for $x \in [ 5/6, 1 ]$ the third integrand is bounded by
$$
 \max\{v_{i + 1 }(u) - 1  - v_{i + 1 }(1 - x),  0 \}
 \leq 0
$$
since $ - v_{i+1}(1-x) \leq - v_{i+1}(1/6)$, and by the induction assumption, we see that the left maximand
$v_{i + 1 }(u) - 1  - v_{i + 1 }( 1/6 )$  is also non-positive. So, at last we have
$$
v_{i} ( u ) - v_{i} (1/6)  \leq 5/6 < 1,
$$
and this completes the proof of condition \eqref{it:UB-Bellman-u-1/6}.

\bigskip

The last claim of Proposition \ref{pr:f-lower-bound} is that
$1/6 \leq f^*_{i,n}(y) $ for all $y \in [0,1]$ and all $1 \leq i \leq n - 2$, $n \geq 3$.
If $y \in [1/6,  1]$ this bound is trivial since $y \leq f^*_{i,n}(y)$ for
all $1\leq i\leq n$.
If $y\in [0, 1/6)$, then the inequality \eqref{it:LB-Bellman-u-5/6} gives us that
$1 < v_{i+1, n}(y) - v_{i+1, n}(5/6)$ for all $1\leq i \leq n-2$, so that the definition
of $ f^*_{i,n}(y) $ in \eqref{eq:f-FINITE-star} gives the required
lower bound. This completes the proof of  Proposition \ref{pr:f-lower-bound}.
\end{proof}

\subsection*{Proof of the Upper Bound --- The Last Step}
We now have the all the tools that we need to prove that there is a constant $C < 11 - 4\sqrt{2} \sim 5.343$ such that
$$
\E[A^o_n(\pi^*_n)] \leq (2 - \sqrt{2}) n  + C \quad \text{ for all } n\geq 1.
$$
We first note that the bound is trivial for $n=1$ and $n=2$.
For $n\geq 3$ we let $\{f^*_{1,n}, \ldots, f^*_{n,n}\}$ denote the optimal threshold functions determined
by recursive solution of the Bellman equation \eqref{eq:Bellman-FINITE-flipped} for the finite horizon problem with sample size $n$.
We will use the first $n-2$ of these functions
to construct a suboptimal selection policy for the geometric sample size problem. From the suboptimality of this policy
we will obtain an inequality that will lead to our upper bound.

\subsection*{Construction of a Suboptimal Policy for the Infinite Horizon Problem}

We now consider the infinite horizon problem, and, as before,
we let $\{X_1,X_2, \ldots\}$ denote the sequence of observations.
Here is our selection process:

\begin{itemize}
\item
We let $T_0$ denote the index of the first observation in the sequence that falls in the interval $[5/6,1]$. We select that observation
as first element of our subsequence and we set $Y_{T_0}=1-X_{T_0}$.
We note that  $Y_{T_0}$ has the uniform distribution in  $[0, 1/6]$.
\item
Next we use the functions $\{f^*_{1,n}, \ldots, f^*_{n-2,n}\}$ to decide which of the next $n-2$ observations are to be selected.
Specifically, we make our
$i$'th selection in the series if  $X_{T_0 + i}\geq f^*_{i,n}(Y_{T_0 + i - 1})$ where as usual the $Y_{T_0 + i}$ are defined by the recursion
\begin{equation*}
Y_{T_0 + i} =
\begin{cases}
    Y_{T_0 + i - 1} & \text{if $X_{T_0 +i}< f^*_{i,n}(Y_{T_0 + i - 1})$}\\
    1 - X_{T_0 + i} & \text{if $X_{T_0 +i}\geq f^*_{i,n}(Y_{T_0 + i - 1})$}.
\end{cases}
\end{equation*}
Here one should recall that by Proposition \ref{pr:f-lower-bound} we have
$1/6 \leq f^*_{i,n}(Y_{T_0 + i - 1})$ for $1 \leq i \leq n-2$, so we have
$0 \leq Y_{T_0 + i} \leq 5/6$ for  $1 \leq i \leq n-2$.

\item We will now show how our selection process can be repeated in a stationary way. For $k=0, 1, 2, \ldots$ we proceed as follows:
\begin{enumerate}[1.]
    \item If $Y_{T_k+n-2} \in (1/6, \, 5/6 ]$, then we let $$\tau_k = \inf\{i \geq 1: X_{T_k + n - 2 + i} \geq 5/6\},$$
          and we select the observation $ X_{T_k + n - 2 + \tau_k}$. We note that the random variable
          $ Y_{T_k + n - 2 + \tau_k} = 1 - X_{T_k + n - 2 + \tau_k} $
          is uniformly distributed on $[0, 1/6]$.
    \item If $Y_{T_k + n - 2} \leq 1/6$ , then we simply let $\tau_k = 0$, and we again note that
    $Y_{T_k + n - 2+ \tau_k}$ is uniformly distributed on $[0, 1/6]$.
    \item We set $T_{k+1} = T_k + n - 2 +\tau_k $ and set $k = k + 1$.
    \item Just as in the second bullet, we use the functions $\{f^*_{1,n}, \ldots, f^*_{n-2,n}\}$ to decide
          which observations to select from $\{X_{T_k + 1},X_{T_k+2}, ..., X_{T_k + n - 2}\}$.
          At time $T_k + n - 2$ we are left with some value $Y_{T_k + n - 2}$,
          and we return to Step 1 of this bullet.
\end{enumerate}
\end{itemize}

\subsection*{Analysis of the Policy}

The suboptimal policy we constructed provides us with an increasing
sequence of stopping times $0 < T_0 < T_1 < T_2 < \cdots$ such that
the times $\{T_k: k \geq 1\}$ are regeneration times for the process
$\{Y_i: i\geq T_0\}$.
Moreover, we also have an i.i.d. sequence of stopping times $\{\tau_k: k\geq 1\}$
with distribution
$$
\tau_k \stackrel{d}{=}
\begin{cases}
     0            & \text { if } Y_{ T_0+n-2 } \leq 1/6 \\
     \inf\{i \geq 1: X_{i} > 5/6\} & \text { if } Y_{T_0+n-2} > 1/6.
\end{cases}
$$
These regeneration times $\{T_k: k\geq 1\}$ can be written as function
of the stopping times $\{\tau_k: k\geq 1\}$; specifically, we have
\begin{equation}\label{eq:Tk}
T_k = T_0 + (n-2) k + \sum_{j=1}^{k} \tau_j.
\end{equation}

For any pair $(T_k, Y_{T_k})$, $1\leq k < \infty$, the number $ r(T_k, Y_{T_k})$ of selections made from
$\{X_{T_k+1}, \ldots, X_{T_k+n-2}\}$ is then given by the sum
$$
 r(T_k, Y_{T_k}) \stackrel{\rm def}{=} \sum_{i=1}^{n-2} \1(X_{T_k + i} \geq f^*_{i,n}(Y_{T_k + i - 1})).
$$

For each $0< \rho < 1$, the selection process described gives us a feasible policy that lower bounds
the expected length -- $\E[A_N^o( \pi^*)]$ -- of the alternating subsequence selected by an optimal policy form a sample of
geometric size.

Moreover, if for discounting purposes we view the number of selections $r(T_k, Y_{T_k})$ as being counted all
at time $T_k + n - 2$, then we obtain a lower bound  for the expected value achieved by
our suboptimal policy. We therefore have the bound
\begin{equation} \label{eq:suboptimality}
\E\left[\sum_{k=0}^\infty \rho^{T_k + n - 2}  r(T_k, Y_{T_k}) \right]
\leq \E[A_N^o( \pi^*)].
\end{equation}

We now note that $T_0$ and $Y_{T_0}$ are independent, and we also note that
for each $k\geq 1$, the post-$T_k$ process  $\{Y_{T_k + i}: i\geq 0\}$ is independent of $T_k$. Consequently, we have the factorization
\begin{equation}\label{eq:independence}
\E\left[\rho^{T_k + n - 2} r(T_k, Y_{T_k}) \right]
= \E\left[\rho^{T_k + n - 2}\right] \E \left[ r(T_k, Y_{T_k}) \right]
\quad \text{ for all } k\geq 0,
\end{equation}
and since $T_k$ is a regeneration epoch we also have
$$
\E \left[ r(T_k, Y_{T_k}) \right] = \E \left[ r(T_0, Y_{T_0}) \right]
\quad \text{ for all } k\geq 0.
$$
For $Y_{T_0} = y \in [0, 1/6]$ we recall the identity \eqref{eq:v1n} and we observe that
$$
v_{1,n} (y) - 2 \leq \E[r({T_0}, Y_{T_0}) | Y_{T_0} = y]
$$
since the policy of the right-hand side agrees with the policy of the left-hand side for the first $n-2$ observations,
and the policy of the right-hand side never selects the last two.

The monotonicity of $v_{1,n} (\cdot)$ and the inequality
\eqref{it:UB-Bellman-u-1/6} of Proposition \ref{pr:f-lower-bound}
then give us the lower bound
\begin{equation*}
\E[A^o_n(\pi^*_n)] - 3 = v_{1,n}( 0 ) - 3 \leq \E[r({T_0}, Y_{T_0}) | Y_{T_0} = y] \quad \text{ for all } 0 \leq y \leq 1/6,
\end{equation*}
so by recalling that $0\leq Y_{T_0} \leq 1/6$ and taking total expectations
we see that
$$
\E[A^o_n(\pi^*_n)] - 3 \leq \E[ r({T_0}, Y_{T_0}) ].
$$
The factorization \eqref{eq:independence} then gives us the bound
$$
\E\left[\rho^{T_k + n - 2}\right] \left( \E[A^o_n(\pi^*_n)] - 3 \right)
\leq \E\left[\rho^{T_k + n - 2} r(T_k, Y_{T_k}) \right]
\quad \text{ for all } k \geq 0.
$$
If we now sum over $k$, use the representation \eqref{eq:Tk} and use the suboptimality condition \eqref{eq:suboptimality},
then we have
\begin{equation}\label{eq:lower-bound-finite-optimal}
\left( \E[ A^o_n(\pi^*_n)] - 3 \right) \E\left[ \sum_{k=0}^\infty \rho^{T_0 + (n - 2) (k + 1) + \sum_{j=1}^{k}\tau_j} \right]
\leq \E[ A^o_N(\pi^*)].
\end{equation}

We now note that $T_0$ is also independent from the random variables $\{\tau_k: k\geq 1\}$,
and we recall that the $\tau_k$'s are i.i.d., so
$$
\E\left[ \sum_{k=0}^\infty \rho^{T_0 + (n - 2) (k + 1) + \sum_{j=1}^{k}\tau_j} \right]
= \E \left[ \rho^{T_0} \right] \sum_{k=0}^\infty \rho^{(n - 2)(k+1) } \, \E \left[ \rho^{\tau_1} \right]^{k}.
$$
Since $x \mapsto \rho^x$ is convex,
Jensen's inequality tells us that $\rho^{\E T_0} \leq \E[\rho^{T_0}]$ and that
$\rho^{\E\tau_1} \leq \E[\rho^{\tau_1}]$, so we have
\begin{equation*}\label{eq:Jensen}
\rho^{\E T_0 + n - 2}\sum_{k=0}^\infty  \left(\rho^{n - 2 + \E\tau_1} \right)^{k}
\leq
\E[\rho^{T_0}]  \sum_{k=0}^\infty \rho^{(n - 2)(k+1) } \, \E[\rho^{\tau_1}]^{k}.
\end{equation*}
The left-hand side is an easy geometric series, and
by substitution in equation \eqref{eq:lower-bound-finite-optimal}
we obtain the crucial bound
$$
\E[A^o_n(\pi^*_n)]
\leq
 3 + \frac{1 - \rho^{n - 2 + \E\tau_1}}{\rho^{\E T_0 + n - 2 }} \, \E[A^o_N(\pi^*)].
$$
From the explicit formula for $\E[A^o_N(\pi^*)]$ in Theorem \ref{th2:geometric-mean}
we then have
$$
\E[A^o_n(\pi^*_n)]
\leq
3 + \frac{(1 - \rho^{n - 2+ \E\tau_1})(3 - 2\sqrt{2} - \rho + \rho\sqrt{2})}{\rho^{\E T_0 + n - 1 } (1 - \rho)}.
$$
The bound above holds for all $0< \rho < 1$, so by letting
$\rho \uparrow 1$ we obtain
$$
\E[A^o_n(\pi^*_n)]  \leq 3 + (2 - \sqrt{2}) (n - 2 + \E\tau_1) < (2 - \sqrt{2}) n + 11 - 4\sqrt{2}
$$
since $\E[\tau_1] <  6$. This completes the proof of the upper bound.

\section{Observations on Methods and  Connections}

Our principal goal has been to provide a reasonably definitive solution of a concrete problem of sequential optimization. Still,
the natural expectation is that the solution of such a problem should
also offer some novel methodological perspective. Here we began by exploiting the well-known idea of passing to the
infinite horizon problem, but less commonly (and somewhat doggedly) we made the
trek back from the infinite horizon problem to the finite horizon problem. In retrospect, that trek had elements of
inevitability to it, but it also had surprises.

In a natural and easy way the policy for the infinite horizon problem gave us a lower bound for the finite horizon problem, but our first
surprise was the discovery (at first numerically) that the lower bound was so close to optimal.
There was also something natural about the upper bound for the finite horizon problem, though at first we argued it by contradiction. The idea was
that if we had a policy for finite horizon that was ``too good" then one should be able to concatenate that policy to give a
policy for the infinite horizon problem that would do better than our known optimal policy. The resulting
contradiction would then provide an upper bound.

This three-step process would seem to be applicable to many problems of optimal selection, though, from the details of our proof, it is clear that
special features must be exploited. For example, without obtaining four relations in Lemma \ref{lm:4conditions}, we would not have
been able to solve the infinite horizon problem.
Three of these relations were straightforward, but the critical fourth relation still seems ``lucky." We are also fortunate that symmetry
relations simplified our Bellman equations. These simplifications have an intuitive basis from the alternating nature of the problem, but
it seems fortuitous that these relations could be made rigorous by inductions (of several kinds) on the Bellman equation.

There are many problems where one would like to go from the infinite horizon problem to the finite horizon problem, but one especially attractive
is that of the optimal on-line selection of a monotone subsequence from a sample of independent observations. Here one knows the
asymptotic behavior of the means for both finite samples \cite{SamSte:AP1981} and random samples --- including geometric sized samples ---
(\citename{Gne:JAP1999} \citeyear*{Gne:JAP1999,Gne:IMS2000}). Most notably, in the infinite horizon case
one has a precise understanding of the variance and even a central limit theorem
(\citename{BruDel:SPA2001} \citeyear*{BruDel:SPA2001,BruDel:SPA2004}).
It would be quite interesting to know if such an analogous CLT can be obtained under the finite horizon
formulation.

\bibliographystyle{agsm}

\end{document}